\documentclass[a4paper,12pt]{article}
\usepackage{fullpage} 
\usepackage[latin2]{inputenc}
\usepackage[magyar,english]{babel}

\usepackage{tikz}

\usepackage{amsmath}
\usepackage{amsthm}
\usepackage{amssymb}
\usepackage{amsrefs}

\newtheorem{thm}{Theorem}
\newtheorem{conj}{Conjecture}
\newtheorem{example}{Example}
\newtheorem{lem}{Lemma}
\newtheorem{defi}{Definition}
\newtheorem{question}{Question}

\theoremstyle{remark}
\newtheorem*{remark}{Remark}
\providecommand{\keywords}[1]{\textbf{\textit{Keywords ---}} #1}

\begin{document} 

\title{On the 1-switch conjecture in the Hypercube and other graphs}

\author{Daniel Soltész\\ Department of Computer Science and Information Theory, \\ Budapest University of Technogoly and Economics}

\maketitle

\begin{abstract}

Feder and Subi conjectured that for any $2$-coloring of the edges of the $n$-dimensional cube, we can find an antipodal pair of vertices connected by a path that changes color at most once. We discuss the case of random colorings, and we prove the conjecture for a wide class of colorings. Our method can be applied to a more general problem, where $Q_n$ can be replaced by any graph $G$, the notion of antipodality by a fixed automorphism $\phi \in Aut(G)$. Thus for any $2$-coloring of $E(G)$ we are looking for a pair of vertices $u,v$ such that $u= \phi(v)$ and there is a path between them with as few color changes as possible. We solve this problem for the toroidal grid $G=C_{2a} \square c_{2b}$ with the automorphism that takes every vertex to its unique farthest pair. Our results point towards a more general conjecture which turns out to be supported by a previous theorem of Feder and Subi.


\end{abstract}

\keywords{hypercube, n-cube, edge coloring, labelling, antipodal, automorphism, discrete torus }

\section{Introduction}

The graph of the $n$-dimensional cube $Q_n$, has vertex set $\{0,1\}^n$, two vertices are adjacent if and only if they differ at precisely one coordinate. We say that an edge faces the $i$-th direction if the endpoints of the edge differ in the $i$-th coordinate. Two vertices are antipodal if they differ in every coordinate. Two edges are antipodal if the endpoints of one are the antipodal pairs of the endpoints of the other. A $k$-coloring of the edges of $Q_n$ is a function from $E(Q_n)$ to $\{1,\ldots, k\}$. In this paper we will investigate $2$-colorings, so we will refer to the two colors as {\em red} and {\em blue}. We say that a $2$-coloring is antipodal if antipodal edges receive different colors. We begin with the chronologically first conjecture due to Norine. 

\begin{conj}[S. Norine \cite{norine}] \label{conj:norine}
If $E(Q_n)$ is $2$-colored antipodally, then there is a pair of antipodal vertices connected by a monochromatic path. 
\end{conj}

While proving Conjecture $1$ for a large class of colorings, Feder and Subi formulated the $1$-switch version that does not require the coloring to be antipodal. 

\begin{conj}[T. Feder, C. Subi \cite{antipod}] \label{conj:federsubi} 
 For every $2$-coloring of $E(Q_n)$ there is a pair of antipodal vertices joined by a path that changes its color at most once.

\end{conj} 

Note that even the question whether $o(n)$ color changes suffice is still open \cite{leaderi}. 
The smallest example of a coloring where we need a color change is a properly edge-colored $Q_2$ (a $4$-cycle). Feder and Subi call colorings of $E(Q_n)$ without properly colored $4$-cycles \textit{simple} colorings. They proved Conjecture \ref{conj:federsubi} in the case of simple colorings, without a color change.

\begin{thm}[T. Feder, C. Subi \cite{antipod}] \label{simple} 
For every simple $2$-coloring of $E(Q_n)$ there is a pair of antipodal vertices connected by a monochromatic path. 
\end{thm}

Thus in some sense, the properly colored $4$-cycles are the reason of the color change. Leader and Long conjectured the $n$-long-path version of the previous conjectures. While Conjecture \ref{conj:federsubi} is stronger than Conjecture \ref{conj:norine} (see \cite{antipod}), if we require shortest ($n$-long) paths, they become equivalent (see \cite{leaderi}). 

\begin{conj}[I. Leader, E. Long] \cite{leaderi} \label{conj:leader} The following equivalent statements hold: 
\begin{itemize} 
\item In every $2$-coloring of $E(Q_n)$, there are antipodal vertices connected by a path of length $n$ which changes its color at most once. 
\item In every antipodal $2$-coloring of $E(Q_n)$ there are antipodal vertices connected by a monochromatic path of length $n$.
\end{itemize}
\end{conj}

To state our main result we define the component graph. While we feel that the definition of the component graph is natural in the first place, we present more motivating thoughts in the next section where we discuss random colorings.

\begin{defi} \label{comp} Let $G$ be a graph, and $c$ a $2$-coloring of $E(G)$. The pair $(G,c)$ determines a graph we call the component graph $Comp(G,c)$. The vertices of $Comp(G,c)$, are the monochromatic components of $G$ with coloring $c$. Two vertices are connected in the component graph if the corresponding monochromatic components have at least one common vertex in $G$. 
\end{defi}

$Comp(G,c)$ is always simple and bipartite (as we are interested in $2$-colorings). It is connected if and only if $G$ is connected. The most important property of the component graph is that for any path in $G$, there is a corresponding walk in $Comp(G,c)$. The length of the walk is exactly the number of color changes of the path.  

\begin{thm}[Main theorem] \label{thm:main} Let $G$ be a graph, $\phi \in Aut(G)$ an automorphism, $c$ a $2$-coloring of $E(G)$ and $0 \leq k$ an integer. If the length of the longest cycle in $Comp(G,c)$ is less than  $2k+3$, then there is a vertex $u \in V(G)$ such that there is a path connecting $u$ and $\phi(u)$ that changes colors at most $k$ times. 
\end{thm}


The relevance of our theorem in the view of Conjecture \ref{conj:federsubi} can be seen by choosing $G=Q_n$ and $\phi$ the automorphism that takes every vertex to its antipodal pair. The case of $k=0$ is of particular interest since then the assumption becomes that $Comp(Q_n,c)$ has to be a tree. In \cite{antipod} Feder and Subi proved that if the coloring is simple, the component graph is a tree, thus Theorem \ref{thm:main} is a generalization of Theorem \ref{simple}. It is a simple corollary of a theorem of Bollobás that the ratio of the number of colorings where $Comp(Q_n,c)$ is a tree to the number of all colorings of $E(Q_n)$ tends to one as $n$ tends to infinity. 

The property, that $Comp(Q_n,c)$ is a tree is (in an equivalent form) used by Feder and Subi to prove Theorem \ref{simple}. Their method is tailored to $Q_n$ and does not seem to work without the additional assumption of the simplicity of the coloring. Our results naturally point towards the generalization that we can use any automorphism of $Q_n$.


We explore the analogous problem for other graphs and automorphisms as well. If $G$ is a graph and $\phi \in Aut(G)$, let $d(G,\phi)$ denote the smallest integer such that no matter how we $2$-color $E(G)$ we will always find a pair of vertices $(u,v)$ such that $\phi(u)=v$ and there is a path connecting $u$ and $v$ that changes colors at most $d(G,\phi)$ times. Let us denote the maximum of $d(G,\phi)$ over all automorphisms of $G$ by $D(G)$. We conjecture the following generalization of Conjecture $2$. 

\begin{conj} \label{conj:weak} 
Let $G$ be the Cartesian product of cycles of even length $a_1, \ldots , a_k$ not all equal to $2$. If $\phi_1, \phi_2, \ldots, \phi_k, \phi$ are the automorphisms, that take every vertex to the unique farthest vertex in the graphs $C_{a_1}, C_{a_2}, \ldots, C_{a_k}, G$ respectively, then $d(G,\phi) = \max_{i}{d(C_{a_i},\phi_i)}=\max_{i}{D(C_{a_i})} =\max_{i}{\frac{a_i}{2}-1}$.
\end{conj}

Here a $2$-cycle denotes the complete graph on two vertices. The assumption that not all cycle lengths are equal to $2$ is to avoid trivial counterexamples. Conjecture \ref{conj:weak} is a generalization of Conjecture \ref{conj:federsubi} as $Q_n$ can be written as the Cartesian product of $4$-cycles, and a $2$-cycle if $n$ is odd. We prove Conjecture \ref{conj:weak} for $k=2$ ie. the rectangular grid graph on the torus. 

The paper is organized as follows. In the next section we discuss random colorings. The third section contains the proof of the Theorem \ref{thm:main}, a similar result in the case of $Q_n$ where we are concerned about induced cycles of $Comp(Q_n,c)$, and the limitation of these theorems to prove Conjecture \ref{conj:federsubi}. The fourth section contains the proof of Conjecture \ref{conj:weak} in the case of $k=2$. Finally in section five we present some open questions similar to Conjecture \ref{conj:federsubi}.

\section{Random colorings}

Let $Q_{n,p}$ denote the random subgraph of $Q_n$ that contains every edge of $Q_n$ independently with probability $p$. The connectivity of $Q_{n,p}$ is extensively studied. The following two theorems are relevant in view of Conjecture \ref{conj:federsubi}. 

\begin{thm}[Erdős, Spencer \cite{erdos}]   

$$ \mathbb{P}(Q_{n,p} \, \text{is connected})= 
\left \{
  \begin{tabular}{cc}
  $0$ & if $p <1/2$ \\
  $\frac{1}{e}$ & if $p=1/2$  \\
  1 & if $p > 1/2$ 
  \end{tabular} \right.
 $$
\end{thm}

\begin{thm}[Bollobás \cite{bollobas}]  The probability that $Q_{n,1/2}$ contains a connected component of size at least $2$ but at most $2^{n-1}$ tends to zero as $n$ tends to infinity. 
\end{thm}

This means that for $p=1/2$ the usual reason of disconnectedness is the presence of isolated vertices. Thus in a typical coloring of the $n$-cube, for both colors, there is a giant monochromatic component with more than half of the vertices in it (and the other monochromatic components are isolated vertices). In these cases, the conjecture holds without a color change, and the reason for it is the size of the largest monochromatic component. In general we cannot expect this. Consider the following coloring: 

\begin{example}[Directional coloring] \label{dircol} For $i \leq k$ let the edges of $Q_{2k}$ be colored red if the edges face the $i$-th direction, and blue otherwise. 
\end{example}

In this coloring every monochromatic connected component is a $k$-dimensional subcube. Thus the size of the maximal monochromatic component can be as small as the square root of the number of vertices. One might suspect that while in general the size of the maximal monochromatic component is not large enough, but it is not a problem, as we can have a color change. So maybe there is always a pair of antipodal vertices connected by a one-switch path, such that one of them is in a monochromatic component of maximal size. This is not true either as it can be seen in the following example. 

\begin{example} Let $n=m+k$, color the edges of $Q_n$ as follows: The edges facing in the first $m/2$ directions are red except when the last $k$ coordinates are all ones. The edges facing in the second $m/2$ directions are red only when the last $k$ coordinates are all zeros. An edge facing the last $k$ directions is colored red if: when we change the coordinate corresponding to that edge to one, the resulting vector will have an even number of ones in the last $k$ coordinates. All the other edges are blue. 
\end{example}

We will set $k=m/4$, but the construction is visualized easier if we use $m$ and $k$: Ignore the first $m$ coordinates and we have a $k$-dimensional cube that is colored according to its levels. Now taking into account the first $m$ coordinates, every vertex of the $k$-cube becomes an $m$-cube, and is colored by a directional coloring, except for the all-zero and the all-ones vertex of the $k$-cube. They become monochromatic red and blue $m$-cubes respectively. Thus there are two monochromatic components of size $2^m$, and all the others are of size at most $2^{m/2+1}  \binom{k}{k/2}$. Thus elementary estimates show that if $k < m/4$ and $k$ is large enough, we have that the largest monochromatic components are the ones that form the monochromatic $m$-cubes. But because of the coloring of the last $k$ coordinates, we need at least $k-1$ color changes (which is $\Omega(n)$) from any vertex of the largest (red or blue) component to its antipodal pair. So we see that the size of the monochromatic components is not sufficient to guarantee antipodal pairs joined by a one-switch path. We have to take into account the way these components connect to each other. A natural way to do this is to define the component graph, see Definition \ref{comp}. 

It is clear that the component graph is connected if and only if $G$ is connected, and it is always bipartite. For any vertex in $G$ there are exactly two vertices corresponding to it in $Comp(G,c)$, the two monochromatic components containing that vertex. The most useful property of the component graph is that the minimal number of color changes to get from $u$ to $v$ in the colored graph is the same as the minimum distance between components corresponding to these vertices in the component graph. It follows from the above mentioned theorem of Bollobás, that the component graph of $Q_n$ is typically a tree, but the directional coloring is an example showing that it is not always the case: When $Q_n$ is colored directionally, $Comp(Q_n,c)$ is a complete bipartite graph.

\section{$Q_n$ and the component graph}

In this section we prove Theorem \ref{thm:main}. Roughly speaking it shows that for the need of high number of color changes, we need long cycles in $Comp(G,c)$. For convenience we restate Theorem \ref{thm:main} 
\setcounter{thm}{1}
\begin{thm} 
Let $G$ be a graph, $\phi \in Aut(G)$ an automorphism, $c$ a $2$-coloring of $E(G)$ and $0 \leq k$ an integer. If the length of the longest cycle in $Comp(G,c)$ is less than  $2k+3$, then there is a vertex $u \in V(G)$ such that there is a path connecting $u$ and $\phi(u)$ that changes colors at most $k$ times. 
\end{thm}
\setcounter{thm}{3}
 \begin{proof}
The vertices of $Comp(G,c)$ are subsets of $V(G)$, thus for any $a \in V(Comp(G,c))$ we can define its image under $\phi$ as follows: $\phi(a):= \{ \phi(x) \, | \, x \in a \}$. We denote the set of components of the image of $a$ under $\phi$ by 
$$S(a):= \{ b \in V(Comp(G,c)) \, | \, \exists x \in a : \phi(x)\in b \}.$$
Having $u$ and $\phi(u)$ in the same monochromatic component $a$, is equivalent to $a \cap S(a) \neq \emptyset$. Having a path between $u$ and $\phi(u)$ that changes its color at most $k$ times is equivalent to $B_k(a) \cap S(a) \neq \emptyset$ where $B_k(a)$ is the closed ball of radius $k$ centered at $a$. Our aim is to show that there is such an $a \in V(Comp(G,c))$. To do this we need two easy properties of $S(a)$. 

\begin{lem} \label{properties} The following statements hold.
\begin{enumerate} 
\item If $a \in V(Comp(G,c))$, then the vertices of $S(a)$ span a connected subgraph. 
\item If $a,b\in V(Comp(G,c))$ are neighbors, then $S(a) \cap S(b) \neq \emptyset$.
\end{enumerate}
\end{lem}
\begin{proof}
The first statement holds by the following argument. By the definition of $S(a)$, for any two vertices $c_1,c_2 \in S(a) \subseteq V(Comp(G,c))$, there are two vertices $x_1,x_2 \in \phi(a) \subseteq V(G)$ such that $x_i \in c_i$ for $i=1,2$. As $\phi(a)$ is connected there is a path from $x_1$ to $x_2$ in $\phi(a)$, and the components of the vertices of this path give a path from $c_1$ to $c_2$ in $S(a)$.

The second statement follows because if $a$ and $b$ are neighbors in $Comp(G,c)$, then by the definition of $Comp(G,c)$, there is a vertex $x \in a \cap b$ in $G$. Thus the  components containing $\phi(x)$ are both in $S(a)$ and $S(b)$. 
\end{proof}


Our strategy is as follows, we will define $a_0,a_1 \ldots \in V(Comp(G,c))$. Let $X_i$ denote the connected component of $V(Comp(G,c)) \setminus B_k(a_i)$ containing $S_(a_i)$. For every $i$ we will either have that $B_k(a_i) \cap S(a_i) \neq \emptyset $ or $X_i \subset X_{i-1}$. This way we are done since for every $i$ we have that $X_i \neq \emptyset$. During the following part of the proof, see Figure \ref{figureone}.

Suppose that we already defined $a_0,\ldots,a_{i}$ and $B_k(a_{i}) \cap S(a_{i}) = \emptyset$. We will define $a_{i+1}$ such that either $a_{i+1}\cap S(a_{i+1}) \neq \emptyset$ or $X_{i+1} \subset X_{i}$. Let us denote the vertices of the neighborhood of $B_k(a_{i})$ in $X_{i}$ by $H_i := h_{i,1},\ldots,h_{i,j}$.  The set of vertices in $H_i$ forms a cut and is also an independent set as they are of the same distance (exactly $k+1$) from $a_{i}$ ($Comp(G,c)$ is bipartite). There cannot be two vertex disjoint paths connecting $a_{i}$ and $H_i$ since two such paths would result in a cycle of length at least $2k+4$ by the following argument: Take the shortest two of such paths, they cannot use any edges from the subgraph spanned by $X_i$. The concatenation of these paths is a path of length at least $2k+2$ connecting vertices $h_{i,u}$ and $h_{i,v}$. Because $X_i$ is connected, there is a path of length at least $2$ using only edges in $X_i$ connecting $h_{i,u}$ and $h_{i,v}$. As these paths are vertex disjoint we can form a cycle of length at least $2k+4$, a contradiction. Thus by Menger's theorem there is a vertex $s_i$ such that every path from $a_i$ to $H_i$ contains $s_i$.

\begin{figure}
\centering
\begin{tikzpicture}
 \draw[thick] (6.9,-0.2) arc (140:220:3)node[anchor=west] {$B_k(a_i)$};
 \draw[thick] (5,0) arc (140:220:3.5)node[anchor=west] {$B_{k+1}(a_i)$};
\filldraw[black] (10,-2) circle (2pt)node[anchor=west] {$a_i$};
\filldraw[black] (8,-1.8) circle (2pt)node[anchor=south] {$s_i$};
 \filldraw[black] (5.8,-1.0) circle (2pt)node[anchor=east] {$h_{i,1}$};
 \filldraw[black] (5.8,-1.5) circle (2pt)node[anchor=east] {$h_{i,2}$};
 \filldraw[black] (5.8,-2.1) circle (0.5pt);
 \filldraw[black] (5.8,-2.25) circle (0.5pt);
 \filldraw[black] (5.8,-2.4) circle (0.5pt);
 \filldraw[black] (5.8,-3.0) circle (2pt)node[anchor=east] {$h_{i,j}$};
\draw[thick] (5.5,-2) ellipse (0.6 and 1.8); 
\draw[thick] plot [smooth,tension=1.2] coordinates {(5.5,-0.2)(3.4,-0.3)(2.2,-0.9)(4.2,-1.5)(1,-2)(3,-4)(5.5,-3.8)}; 
\node($X_i$) at (2,-4){$X_i$};
\node($H_i$) at (5.8,0){$H_i$};
\node($Sa_i$) at (3,-2.7){$S(a_i)$};
\draw[thick] plot [smooth cycle,tension=1] coordinates {(3.5,-2.7)(3.6,-1.9)(3.07,-2.35)(2.3,-2.1)(2.35,-2.7)(2.4,-3.2)(3,-3.05)(3.62,-3.3)}; 
\draw[thick] plot coordinates {(10,-2)(9.4,-1.4)(8.9,-2.1)(8.4,-1.5)(8,-1.8)}; 
\draw[thick] plot coordinates {(8,-1.8)(7.4,-1.4)(7,-1.7) (6.6,-1.25) (5.8,-1.0) }; 
\draw[thick] plot coordinates {(6.6,-1.25)(5.8,-1.5) }; 
 \filldraw[black] (7.45,-2.0) circle (0.5pt);
 \filldraw[black] (7.45,-2.15) circle (0.5pt);
 \filldraw[black] (7.45,-2.3) circle (0.5pt);
\draw[thick] plot coordinates {(8,-1.8)(7.5,-2.8)(7.1,-2.4)(6.6,-3)(5.8,-3.0) };
\end{tikzpicture}

\caption{Choosing $a_{i+1}$} \label{figureone}
\end{figure}
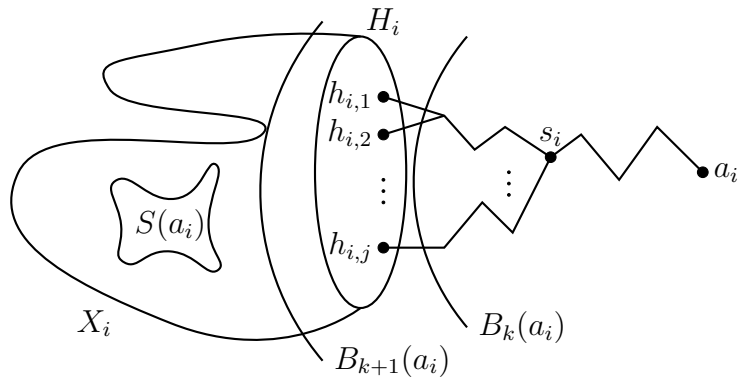

 Let us choose $a_{i+1}$ to be the neighbor of $a_i$ in a shortest path connecting $a_i$ and $s_i$. This way we have that $H_i \subset B_k(a_{i+1})$. If $B_k(a_{i+1}) \cap S(a_{i+1}) \neq \emptyset $ we are done. Otherwise $S(a_{i+1}) \cap H_i = \emptyset$ and by lemma \ref{properties} we have that $S(a_{i+1}) \cap S(a_i) \neq \emptyset$ thus $S(a_{i+1}) \subseteq X_i$. And therefore $X_{i+1} \subset X_i$ also holds as $X_{i+1} \subseteq  H_i = \emptyset$ and the proof is complete. \end{proof}


\begin{remark} Theorem \ref{thm:main} is best possible in the sense that $2k+3$ cannot be replaced by $2k+4$ as the following example shows. Color the edges of a cycle of length $2k+4$ properly, and choose $\phi$ to be the automorphism that takes every vertex to the unique farthest vertex. The component graph is also a cycle of length $2k+4$ and there is no path connecting antipodal vertices with $k$ color changes. 
\end{remark}
\begin{remark}Theorem \ref{thm:main} is not best possible in the sense that there are plenty of component graphs with cycles larger than $2k+4$ where we can  "catch $S(a)$ with $a$". For example consider the complete bipartite graph $K_{m,m}$, no matter how large $m$ is, $B_2(a_0)$ covers the whole graph. Thus our result can be strengthened in some cases by first decomposing $V(G)$ into parts of small diameter. The exact characterization of graphs where we can catch $S(a)$ with $a$ is not known to us. 
\end{remark}
\begin{remark}
The fact that Theorem \ref{thm:main} holds for arbitrary graphs and automorphisms, could in principle be used when trying to prove Conjecture \ref{conj:federsubi} in the following way: Given a coloring of $E(Q_n)$ we would reduce the size of the cycles in $Comp(Q_n,c)$ by deleting edges from $Q_n$. In general, if we delete a full orbit of an edge in $G$ (an antipodal pair in $Q_n$), the same function on the vertices remains an automorphism. Thus it is possible to delete edges of $Comp(Q_n,c)$ by deleting a set of edges (and their antipodal pairs) in $Q_n$. But one shall be careful not to disconnect the resulting graph in such a way, that $x$ and $\phi(x)$ would be in different connected components for all $x$. It would be nice not to disconnect the graph at all. We do not know any example of $Comp(Q_n,c)$ where we can not reduce the size of the maximal cycle to $4$ this way. But we could not prove either that it can always be done. 
\end{remark}

The next theorem shows that in the case of $Q_n$ for the necessity of a high number of color changes we also need long { \em induced} cycles. 

\begin{thm} \label{thm:induced} If $c$ is a $2$-coloring of $E(Q_n)$ such that every path connecting antipodal vertices changes colors at least $k>1$ times, then the component graph $Comp(Q_n,c)$ contains an induced cycle of length at least $2k-2$.
\end{thm}
\begin{proof} Let $x_0, \ldots , x_t$ be vertices of a path in $Q_n$ connecting antipodal vertices with $k$ color changes, and $A := \{a_0, \ldots , a_k \}$ be the corresponding path in $Comp(Q_n,c)$.  For every $i$, let $y_i$ denote the antipodal pair of $x_i$ and $B:= \{b_0, \ldots , b_{m_1} \}$ be the walk in $Comp(Q_n,c)$ associated to the path $y_0, \ldots, y_t$. The intersection of the sets $A \setminus \{a_0,a_k\}$ and $B$ must be empty. Otherwise the whole path $A$ would be in $B_{k-1}(b_j)$ for some $j$ and by definition, we have a vertex in $b_j$ that has its antipodal pair in a component in $A$. This contradicts our assumption that we need $k$ color changes from every vertex to its antipodal pair. Let us denote the shortest path from $a_0$ to $a_k$ in the subgraph of $Comp(Q_n,c)$ induced by $b_0, \ldots,b_{m_1},a_0,a_k $ by $D:=\{a_0=d_0,d_1, \ldots, d_{m_2},d_{m_2+1}=a_k \}$. We have that $k \leq m_2+1$ otherwise we would have a path connecting $a_0$ and $a_k$ of length less than $k$. The vertices $C:=\{a_0,a_1 ,\ldots,a_{k-1},a_k,d_{m_2},d_{m_2-1},\ldots , d_1\}$ form a cycle in $Comp(Q_n,c)$ of length at least $2k$. If there are no chords in $C$ we are done. If there are chords in $C$ we will find a large induced cycle among the subgraphs of $C$. Due to the fact that $a_0$ and $a_k$ are connected by shortest paths, there are no chords connecting them to any other vertex of $C$, moreover the two subgraphs spanned by $\{a_1, \ldots , a_k\}$ and $\{d_1, \ldots , d_{m_2}\}$ are paths. Thus we can only have chords connecting vertices from $\{a_1, \ldots, a_{k-1}\}$ to vertices in $\{d_1, \ldots, d_{m_2}\}$. But $C$ can not have a chord connecting a vertex in $\{a_2, \ldots, a_{k-2}\}$ to a vertex in  $\{d_1, \ldots, d_{m_2}\}$ as a chord of this type would mean that the whole path $A$ would be in $B_{k-1}(d_j)$ where $d_j$ is the endpoint of the chord. So every chord in $C$ is adjacent either to $a_1$ or to $a_{k-1}$. 

If there exist chords adjacent to $a_1$ and $a_{k-1}$ (see Figure \ref{figuretwo}) we define $d_{a_1}$  and $d_{a_{k-1}}$ to be the neighbors of $a_1$ and $a_{k-1}$ respectively such that the distance between $d_{a_1}$ and $d_{a_{k-1}}$ is minimal in $D$. By this minimality, the cycle $a_1 \ldots, a_{k-1},d_{a_{k-1}},\ldots, d_{a_1}$ is induced, and of length at least $2k-2$ as otherwise we could reach every $a_1, \ldots, a_k$ from $d_{a_1}$ with paths shorter than $k-1$, a contradiction.
\begin{figure}
\centering
\begin{tikzpicture}
 \filldraw[black] (4,4) circle (2pt)node[anchor=south] {$a_0$};
 \filldraw[black] (2.6,3.3) circle (2pt)node[anchor=east] {$a_1$};
 \filldraw[black] (5.1,3.1) circle (2pt)node[anchor=west] {$d_{a_{1}}$};
 \filldraw[black] (2.8,0.5) circle (2pt)node[anchor=east] {$a_{k-1}$};
 \filldraw[black] (4.85,0.6) circle (2pt)node[anchor=west] {$d_{a_{k-1}}$};
 \filldraw[black] (4,0) circle (2pt)node[anchor=north] {$a_k$};
\draw[thick] plot coordinates {(4,4)(2.6,3.3)(3.3,2.6)(2.7,1.9)(3.4,1.2)(2.8,0.5)(4,0)}; 
\draw[thick] plot [smooth,tension=1.2] coordinates {(4,4)(5.5,2)(4,0)}; 
\draw[thick] plot coordinates {(2.8,0.5)(4.5,0.3)}; 
\draw[thick] plot coordinates {(2.8,0.5)(4.73,0.5)}; 
\draw[thick] plot coordinates {(2.8,0.5)(4.85,0.6)}; 
\draw[thick] plot coordinates {(2.6,3.3)(4.5,3.7)}; 
\draw[thick] plot coordinates {(2.6,3.3)(4.85,3.4)}; 
\draw[thick] plot coordinates {(2.6,3.3)(4.95,3.3)}; 
\draw[thick] plot coordinates {(2.6,3.3)(5.1,3.1)}; 
\end{tikzpicture}

\caption{If there are both types of chords.} \label{figuretwo}
\end{figure}
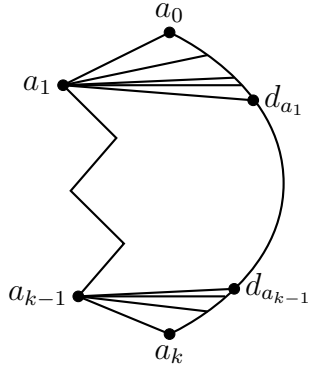

If every chord is adjacent only to w.l.o.g. $a_1$, then let $d_{a_1}$ denote the neighbor of $a_1$ amongst the $d_i$ with the largest index, and the cycle $C'= \{ a_1, a_2,\ldots,a_k,d_{m,1},d_{m_1-1}\ldots, d_{a_1}\} $  is induced and must be at least $2k$ long, otherwise we could reach every $a_1, \ldots, a_k$ from $d_{a_1}$ with a path shorter than $k-1$, a contradiction.   \end{proof}

\begin{remark} We did not use the structure of $Q_n$, only that the order of the automorphism is two. We also feel that the theorem holds with $2k$ instead of $2k-2$, but as the next example shows, this is not that relevant in the view of Conjecture \ref{conj:federsubi}.  \end{remark}

Turning our attention back to Conjecture \ref{conj:federsubi}, while Theorem \ref{thm:main} and Theorem \ref{thm:induced} settle a wide class of special cases. They are insufficient to guarantee even a $o(n)$ bound on the number of color changes. In the case of Theorem \ref{thm:main}, we saw in Example \ref{dircol} that there are colorings of $Q_n$ where $Comp(Q_n,c)$ contains exponentially large cycles. For Theorem \ref{thm:induced} the following example shows that there are colorings of $Q_n$ with induced cycles of length $\Omega(n)$.

\begin{example}
Let $n=2k$ and the edges of $Q_n$ be colored as follows. An edge facing the first $k$ directions is colored red if: when we change the coordinate corresponding to that edge to one, the sum of the first $k$ coordinates of the resulting vector is an odd number. An edge facing the last $k$ directions is colored red if: when we change the coordinate corresponding to that edge to one, the sum of the last $k$ coordinates of he resulting vector is an odd number. All the other edges are blue. 
\end{example} 
The component graph which belongs to this coloring can be understood as follows. If we fix the last $k$ coordinates (no matter how) the resulting cube is colored alternately according to its levels. We can say the same thing if we fix the first $k$ coordinates. Thus a red component contains vertices such that the first $k$ coordinates are always between two consecutive levels in a $k$-dimensional cube, and the second $k$ coordinates are between (possibly some other) two consecutive levels of a $k$-dimensional cube. The blue components also look like this. Thus we can represent a component by a vector of length two: if the component contains vertices between the levels $(a,a+1)$ in the first $k$ coordinates and between the levels $(b,b+1)$ in the second $k$ coordinates, we represent the component with the vector $(a,b)$. Red components are represented with vectors $(a,b)$ such that $a \equiv b \equiv 0 \pmod{2}$, and the blue ones with $a \equiv b \equiv 1 \pmod{2}$. It is easy to check that two components $(a,b)$ and $(c,d)$ are connected by an edge in $Comp(Q_n,c)$ if and only if $|a-c|=1=|b-d|$. Thus the component graph resembles a rectangular grid, and when $k$ is large, it is straightforward to find an induced cycle of length roughly $4k$ in it.

\section{A generalization to torus-like graphs}

Theorem \ref{thm:main} suggests that Conjecture \ref{conj:federsubi} is a question of the following type. Given a graph $G$ and an automorphism $\phi \in Aut(G)$, what is the minimal number $k$ of color changes, such that for every $2$-coloring of $E(G)$, there is a vertex $u$ and a path connecting $u$ to $\phi(u)$ which changes colors at most $k$ times? Let us call this parameter $d(G,\phi)$. It is immediate that for a cycle $C_{2k}$ and for the automorphism $\psi$ that takes every vertex of the cycle to the unique furthest vertex, $d(C_{2k}, \psi)=k-1$. Thus we see that in general this parameter can be arbitrarily large. When trying to explore various graphs and automorphisms, it is natural to look for a group, and consider its Cayley graph (without orientation and coloring) so we have a graph and a considerable amount of its automorphisms. With this in mind, observe that $Q_n$ and $C_{2k}$ are both special cases of cayley graphs of the following groups. Let $a_1, \ldots, a_n$ be integers.   

$$ G(a_1, \ldots , a_n):= \{ x_1,\ldots,x_{n} | \forall i,j : x_ix_j=x_jx_i, x_{i}^{a_i}=1 \} $$

Let us choose for the Cayley graph the natural generating set $ \{ x_1, \ldots, x_n \} $. This way, the parameters $(a_1, \ldots, a_n)$ determine a graph.  We get $Q_n$ with $(2,2,\ldots,2)$, and $C_{2k}$ with $(2k)$. We restrict our attention to even $a_i$ since in the odd case, the furthest vertex is not unique. The multiplication with $x_1^{a_1/2}\cdot \ldots \cdot x_n^{a_n/2}$ is an automorphism which takes every vertex to its antipodal pair (the unique farthest vertex). Note that the graph we get from the vector $(a_1, \ldots, a_n)$ is the Cartesian product of cycles $C_{a_1}, \ldots, C_{a_n}$, with the assumption that $C_2$ is a single edge. While we could just consider the Cartesian products of cycles in the first place, we think that the above mentioned way of arriving to these graphs motivates their study. Note that by using the longest cycle, it is straightforward to color the edges of $C_{a_1} \square \ldots \square C_{a_n}$ such that we need $\max_{i}(a_i/2)-1$ color changes. Observe that the Cartesian product of two edges is $C_4$, so we can think of $Q_n$ as the Cartesian product of $4$-cycles (and an additional edge if $n$ is odd). This leads us to formulate Conjecture \ref{conj:weak}.

\setcounter{conj}{3}
\begin{conj} Let $G$ be the Cartesian product of cycles of even length $a_1, \ldots , a_k$ not all equal to $2$. If $\phi_1, \phi_2, \ldots, \phi_k, \phi$ are the automorphisms, that take every vertex to the unique farthest vertex in the graphs $C_{a_1}, C_{a_2}, \ldots, C_{a_k}, G$ respectively, then $d(G,\phi) = \max_{i}{d(C_{a_i},\phi_i)}=\max_{i}{D(C_{a_i})} =\max_{i}{\frac{a_i}{2}-1}$. \end{conj}

Note that Conjecture \ref{conj:weak} is a generalization of Conjecture \ref{conj:federsubi} and suggests that the reason that we need a single color change is that $Q_n$ is made of $4$-cycles. Also note that when we are not allowed to alternately color these $4$-cycles, Theorem \ref{simple} shows that we can not force a color change, as expected. We prove Conjecture \ref{conj:weak} for $k=2$. 

\begin{thm} \label{thm:torus} If $1 \leq a \leq b \neq 1$ and $G=C_{2a}\square C_{2b}$ is the Cartesian product of two cycles of even length, then there are vertices $u,v \in V(G)$ such that their distance is $a+b$, and there is a path between $u$ and $v$ which changes its color at most  $b-1$ times. 
\end{thm}

\begin{proof}
Let us represent the vertices of $G$ by ordered pairs, $(x,y)$ where $x \in \{1,\ldots,2a \}$ and $y \in \{1,\ldots,2b \}$. The coordinates are understood modulo $2a$ and $2b$ respectively. Thus we are looking for a pair of vertices of the form $(x,y)$ and $(x+a,y+b)$, and a path between them that changes color at most $b-1$ times. Observe that it is enough to show that there are vertices of the form $(x,y)$ and $(x+a,y+a)$ and a path between them that changes colors at most $a-1$-times. Since given such a path, we can get from $(x+a,y+a)$ to $(x+a,y+b)$ using only $b-a$ edges, thus we have at most $b-a$ new color changes resulting in at most a total of $b-1$ color changes as required.  Let us call a path \textit{$j$-ascending diagonal} that starts at $(x,y)$ if it connects vertices $(x,y), (x+1,y+1), \ldots ,(x+j, y+j)$ and is of length $2j$. Similarly we define the \textit{$j$-descending diagonal} that starts at $(x,y)$ to be a path connecting vertices  $(x,y), (x-1,y+1), \ldots ,(x-j, y+j)$ and is of length $2j$. 

We will prove slightly more. Either the coloring is such that every cycle of length four is colored properly, or there are vertices $(x,y)$ and $(x+a,y+a)$, connected by an ascending or a descending diagonal that changes colors at most $a-1$ times. 

Let us call a $j$-ascending diagonal starting at $(x,y)$ \textit{lazy} if it has the fewest color changes amongst all such $j$-diagonals, and the subgraph that is a $(j-1)$-ascending diagonal starting at $(x,y)$ is also lazy. It is called lazy, since it does not change its color, until its necessary. Such a diagonal exists by induction on $j$. For $j=1$ it is trivial. Suppose that for $j-1$ we have a lazy $j-1$-ascending diagonal that starts at $(x,y)$, continue this path with the last two edges of a $j$-ascending diagonal that starts at $(x,y)$ and changes its color the minimal number of times. It is easy to check that the resulting path will be a lazy $j$-ascending diagonal. 

Take a lazy $a$-ascending and a lazy $a$-descending diagonal starting from the vertices of the cycle $(1,1),(2,1), \ldots, (2a,1)$. Let us denote the set of these diagonals $D$. We will bound the cumulative number of color changes of diagonals in $D$. For every color change of such a diagonal we associate a $4$-cycle. If an ascending diagonal that started at $(x,y)$ changes colors at $(x+i,y+i)$ or between $(x+i,y+i)$ and $(x+i+1,y+i+1)$ we say that the diagonal changes colors at the $4$-cycle spanned by $(x+i,y+i),(x+i+1,y+i),(x+i,y+i+1),(x+i+1,y+i+1)$. Thus a $4$-cycle can get $0,1$ or $2$ color changes from an ascending diagonal. We associate a $4$-cycle to every color change of a descending diagonal similarly. Every $4$-cycle has a unique ascending and a unique descending diagonal in $D$ which might contribute to the number of color changes at that cycle, so there can be at most $4$ color changes at a $4$-cycle.

\begin{lem} For every $4$-cycle we associated at most two color changes. 
\end{lem}

\begin{proof} Every diagonal can change its color at most two times at a given $4$-cycle. Thus it is enough to show that given a $4$-cycle, if the unique ascending diagonal within $D$ that passes through it, changes its color two times, then the unique descending diagonal passing through it does not change its color. If a lazy ascending diagonal changes its color two times at the $4$-cycle $(x+i,y+i),(x+i+1,y+i),(x+i,y+i+1),(x+i+1,y+i+1)$, it had to happen the following way: We arrived (w.l.o.g.) with color blue to $(x+i,y+i)$, then the next edge must cause a color change, thus both edges leaving $(x+i,y+i)$ are red. But then we again have to change colors, so the edges arriving to $(x+i+1,y+i+1)$ must have color blue. Thus a lazy decreasing diagonal will not change its colors at this $4$-cycle as it is not necessary. 
\end{proof}

This is already enough to prove the existence of a diagonal of length $a$ with the number of color changes being at most $a$. There are $4a$ diagonals in $D$, and their cumulative number of color changes is at most two times the number of $4$-cycles they pass through. Thus the average number of color changes of diagonals in $D$ is at most $4a^2/4a=a$. 

To go down from $a$ color changes to $a-1$, it is enough to improve our bound on the number of color changes at a single $4$-cycle. Consider the cycles at the start of every diagonal in $D$, namely the cycles spanned by $(j,1),(j+1,1),(j,2),(j+1,2)$ for $j \in \{1,\ldots,2a\}$. Since these are the first $4$-cycles, no diagonal can change its color two times at these. If both the ascending and the descending diagonal have to change its color at such a cycle, it has to be colored properly. Thus if the coloring is such that there is a $4$-cycle that is not colored properly, we can rotate the torus so that it is spanned by the vertices $(1,1),(1,2),(2,1),(2,2)$ and we are done. If every $4$-cycle is colored properly, every edge of the form $((i,j),(i+1,j))$ is colored with the first color, and every other edge is colored with the second. Thus we can connect every pair of vertices by a single color change.
\end{proof}

\section{Open questions}

Since Theorem \ref{thm:main} can be applied to any automorphism, it is natural to ask whether Conjecture \ref{conj:federsubi} is true for any automorphism of $Q_n$.

\begin{question}
Is it true that for any $\phi \in Aut(Q_n)$ and any $2$-coloring of $E(Q_n)$, there is a vertex $u$ such that $u$ and $\phi(u)$ are connected by a path that changes its color at most once?
\end{question}

Feder and Subi proved that if no $4$-cycle is colored properly, we do not need a color change in $Q_n$. Is something similar true in general?

\begin{question}
Let $G$ be the Cartesian product of even cycles of length $a_1, \ldots, a_k$ not all equal to $2$, $\phi$ be the automorphism that takes every vertex to its unique furthest vertex in the graph $G$ and $a_j = \max_{i}{a_i}$. Is it true that if we consider only colorings without properly colored non nullhomotopic $a_j$ cycles, there is always a vertex $u$ such that $u$ and $\phi(u)$ are connected by a path that changes its color less than $d(G,\phi)$ times?
\end{question}

In Conjecture \ref{conj:weak} we restricted ourselves to cycles of even length, to have a similar automorphism that we had at $Q_n$. Is this necessary?

\begin{question}
If $G$ is the Cartesian product of cycles $a_1,\ldots, a_k$ not all from the set $\{2,3,5\}$, is it true that $D(G)= \max_{i}{D(C_{a_i})}$?
\end{question}

The restriction on the sizes of the cycles is to avoid trivial counterexamples. We have that $D(C_2)=D(C_3)=D(C_5)=0$ but it is easy to color $C_i \square C_j$ for $i,j \in \{2,3,5\}$ and give an automorphism in such a way that we need a color change. These counterexamples are caused by the trivial phenomenon that if we would be interested in $m$ colorings, at the Cartesian product of $m$ cycles (even if their length is small) we would need $m-1$ color changes.

Leader and Long asked whether $o(n)$ color changes suffice in $Q_n$. We were unable to answer this question, but we were also unable to find colorings of $Q_n$ where the average number of color changes to get to the antipodal point is larger than $c\sqrt{n}$. Not even when we relaxed the condition that every vertex is trying to get to its antipodal pair. An example where the average number of color changes is $c \sqrt{n}$ is when the cube is colored alternately according to its levels.

\begin{question}
Is it true that for every sequence of colorings of $E(Q_n)$, the average number of color changes to get from a vertex to its antipodal pair is $O(\sqrt{n})$?
\end{question}


\begin{bibdiv}
\begin{biblist}
\bib{bollobas}{article}{
  title={The evolution of the cube},
  author={B. Bollobás},
	journal={Annals of Discrete Mathematics},
	volume={17},
	pages={91-97},
  date={1983},
}
\bib{norine}{book}{
  title={Edge-antipodal colorings of cubes, Open Problem Garden},
  author={M. Devos},
	author={S. Norine},
	date={2008},
}
\bib{erdos}{article}{
  title={Evolution of the $n$-cube},
  author={P. Erdős},
	author={J. Spencer},
	journal={Computers \& Mathematics with Applications},
	volume={5},
	issue={1},
	pages={33-39},
  date={1979},
}
\bib{antipod}{article}{
  title={On Hypercube Labellings and Antipodal Monochromatic Paths},
  author={T. Feder},
	author={C. Subi},
	journal={Discrete Applied Mathematics},
	volume={161},
	pages={1421-1426},
  date={July 2013},
}
\bib{leaderi}{article}{
  title={Long geodesics in the subgraphs of the cube},
  author={I. B. Leader},
	author={Eoin Long},
	journal={\tt arXiv:0706.1234 [math.FA]},
	date={2013 Jan},
}
\end{biblist}
\end{bibdiv}
\end{document}